\documentclass{amsart}
\usepackage[utf8]{inputenc}
\usepackage{amsmath,amssymb,mathtools,amsthm,enumerate}
\usepackage{soul,xcolor}

\newtheorem{thm}{Theorem}[section]
\newtheorem*{mainthm}{Main Theorem}
\newtheorem{propn}[thm]{Proposition}
\newtheorem{lem}[thm]{Lemma}

\newtheorem{cor}[thm]{Corollary}
\theoremstyle{definition}
\newtheorem{defn}[thm]{Definition}
\newtheorem{eg}[thm]{Example}
\newtheorem{q}[thm]{Question}
\theoremstyle{remark}
\newtheorem{rmk}[thm]{Remark}

\newcommand{\Z}{\mathbf{Z}}
\newcommand{\Q}{\mathbf{Q}}
\newcommand{\C}{\mathbf{C}}
\renewcommand{\P}{\mathbf{P}}
\newcommand{\F}{\mathbf{F}}
\renewcommand{\H}{\mathbf{H}}
\newcommand{\Orb}{\mathcal{O}}
\let\le\leqslant
\let\ge\geqslant
\let\bar\overline
\DeclareMathOperator{\codim}{codim}
\DeclareMathOperator{\PrePer}{PrePer}
\DeclareMathOperator{\Per}{Per}
\DeclareMathOperator{\Res}{Res}

\DeclareMathOperator{\End}{End}

\DeclareMathOperator{\Irr}{Irr}
\DeclareMathOperator{\Fix}{Fix}

\DeclareMathOperator{\lcm}{lcm}

\DeclareMathOperator{\PGL}{PGL}

\newcommand\grad\nabla
\newcommand\into\hookrightarrow

\usepackage{biblatex}
\begin{document}

\title[Dynamically improper hypersurfaces]{Dynamically improper hypersurfaces for endomorphisms of projective space}

\author{Matt Olechnowicz}
\address{Department of Mathematics, University of Toronto}
\email{mateusz@math.toronto.edu}
\date{\today}

\begin{abstract}
We introduce a new 
generalization of the notion of preperiodic hypersurface 
and explore some of its basic ramifications.
We also prove that among nonlinear 
endomorphisms of projective space,
those with a periodic critical point are Zariski dense.
This answers a question of Ingram.
\end{abstract}

\maketitle

\section{Introduction}

Let $f$ be a degree-$d$ endomorphism of $\P^n$ defined over an algebraically closed field of characteristic 0.
A subvariety $X$ of $\P^n$ is called \emph{preperiodic} if for every irreducible component $Z$ of $X$ there exist 
distinct non-negative integers $s$ and $t$ such that $f^s(Z) = f^t(Z)$. 
For $d \ge 2$, 
the set $\PrePer(f)$ of 
            irreducible 
zero-dimensional preperiodic subvarieties (i.e.,~preperiodic points) is Zariski dense in 
    every preperiodic subvariety of
$\P^n$,
though if $f$ is 
very
general, then $f$ has no positive-dimensional preperiodic subvarieties other than $\P^n$ itself
(see Fakhruddin \cite[Theorem 5.1]{Fakhruddin2003} and \cite[Theorem 1.2]{Fakhruddin2014}).

Attached to every nonlinear endomorphism $f$ is a canonical hypersurface $C_f$ called its \emph{critical locus}, which is cut out by the Jacobian determinant of any 
            affine 
lift of $f$.
If $C_f$ is preperiodic, then $f$ is said to be \emph{post-critically finite}.
When $n = 1$, post-critically finite maps are relatively common, in the sense that they form a Zariski dense subset of the parameter space ${\End^n_d}$ (see DeMarco \cite[Theorem 1.6]{DeMarco}).
On the contrary, the analogous claim for $n \ge 2$ was 
recently
announced to be false by Gauthier, Taflin, and Vigny \cite[Theorem B]{GVT}.

Thus, the general nonexistence of nontrivial preperiodic subvarieties along with the 
sparseness of post-critically finite maps 
suggests 
that preperiodicity
may be too restrictive a notion in dimension 
greater 
than 1.
To remedy these deficiencies, we propose an alternative generalization of preperiodicity, from points in $\P^1$ to hypersurfaces in $\P^n$.

Consider for the moment an irreducible hypersurface $H$ in $\P^n$.
Its iterated images $H$, $f(H)$, $f^2(H)$, $\ldots$ are again irreducible hypersurfaces in $\P^n$ (assuming $d \ge 1$).
By dimension counting, the intersection of any $n+1$ of them ought to be empty (e.g.,~3 curves in $\P^2$ are rarely concurrent); 
if ever it isn't, we declare $H$ to be \emph{improper}.
We then call a reducible hypersurface \emph{improper} if each of its irreducible components is improper.

We encourage the reader to pause here and verify that
an improper hypersurface in $\P^1$ is just a finite subset of $\PrePer(f)$.
This observation justifies improperness as a 
                    natural 
generalization of preperiodicity. 
But unlike preperiodicity, 
improperness in $\P^n$ does not 
persist upon replacing $f$ by $f^k$ 
(at least in degree 1, cf.~Example \ref{eg:I_not_DI}), 
which prompts us to define a hypersurface to be \emph{dynamically improper} if it is improper under every iterate of $f$ (Definition \ref{defn:improperness}).

Dynamically improper hypersurfaces have many favourable properties in common with preperiodic points;
for instance, they are Zariski dense in their natural parameter space (Corollary \ref{thm:1}).
This abundance, in turn, suggests that the critical locus of an endomorphism stands a good chance at being dynamically improper. 
Indeed, endomorphisms with dynamically improper critical locus (termed \emph{post-critically dynamically improper}) are a direct generalization of the post-critically finite maps of dimension 1 (where dynamical improperness is equivalent to preperiodicity), so by analogy one might expect them to be relatively common.

Our main result confirms this expectation, and answers a question posed by Ingram at the ``Moduli spaces for algebraic dynamical systems'' AIM Workshop in September 2021.

\begin{mainthm}
Let $n \ge 1$ and $d \ge 2$. 
Write $\End^n_d$ for the parameter space of degree-$d$ endomorphisms of $\P^n$.
Then the set of maps with a periodic critical point is Zariski dense in $\End^n_d$.
Thus,
the same holds for the set of post-critically dynamically improper maps (at least assuming $d \ge 3$ when $n \ge 4$).
\end{mainthm}

We note that it is possible to use a result of Berteloot--Bianchi--Dupont \cite[Theorem 1.6]{BBD}, obtained by analytic methods, to deduce the complementary statement that the set of maps with a strictly preperiodic critical point is Zariski dense in $\End^n_d$, at least over $\C$.  We were not aware of this avenue when our research was carried out, and indeed our approach is purely algebraic.

To prove the Main Theorem,
we exhibit infinitely many pairwise distinct hypersurfaces' worth of maps with the desired property.
The hypersurfaces are constructed using elimination theory, 
and their infinitude is established using symmetric powers.

A key technique---indeed, our \emph{only} technique---for proving that a given hypersurface is dynamically improper
is to show that every irreducible component contains a preperiodic subvariety (Proposition \ref{propn:coperiodic}). 
This raises a 
    tantalizing
new ``unlikely intersections'' problem, which we believe should have a positive answer: 

\begin{q} \label{q:non_coperiodic}
Does every 
dynamically improper hypersurface contain a preperiodic point?
\end{q} 

It would also be interesting to study dynamical improperness from an arithmetic perspective.
Inspired by \cite[Theorem 1]{Ingram2}, which says that when $n = 1$ the critical height is a moduli height away from the Latt\`es locus, we ask:

\begin{q}
Is the locus of post-critically dynamically improper maps contained in a set of bounded height plus a proper closed subset of the moduli space $\mathsf{M}^n_d$ of degree-$d$ endomorphisms of $\P^n$?
\end{q}

The author would like to thank Patrick Ingram, Joe Silverman, Jason Bell, Rob Benedetto, Laura DeMarco, and Nicole Looper, as well as the anonymous referees, for their helpful comments on earlier versions of this paper.

\section{Preperiodicity}

The definition of preperiodicity in higher dimensions has some subtleties, which we now address.
(See \cite[\S 2.2.3]{DMLbook} 
for the case of quasiprojective varieties.)
Here, $\Irr(X)$ refers to
the set of irreducible components of the topological space $X$.

\begin{propn}
Let $f : \P^n \to \P^n$ be a morphism and let $X \subseteq \P^n$ be closed.
Each of the following implies the next:
\begin{enumerate}[(a)]
    \item For each $Z \in \Irr(X)$ there exist $s \ne t$ such that $f^s(Z) = f^t(Z)$.
    \item $f^s(X) = f^t(X)$ for some $s \ne t$.
    \item $f^s(X) \subseteq f^t(X)$ for some $s > t$.
    \item $\Orb_f(X) := \bigcup_{i \ge 0} f^i(X)$ is closed.
\end{enumerate}
Moreover, (c) $\Rightarrow$ (b). 
If $X$ is pure, then (b) $\Rightarrow$ (a) and (at least if the ground field is uncountable) (d) $\Rightarrow$ (b).
\end{propn}
\begin{proof}
(a) $\Rightarrow$ (b): 
Let $\Irr(X) = \{Z_1, \ldots, Z_r\}$. Pick $t_i \ge 0$ and $p_i \ge 1$ such that $f^{t_i+p_i}(Z_i) = f^{t_i}(Z_i)$ for all $i$. 
Set $t = \max t_i$ and $p = \lcm p_i$. 
Then $f^{t+p}(X) = f^t(X)$.
(b) $\Rightarrow$ (c): Trivial.
(c) $\Rightarrow$ (b): 
    $f^t(X) \supseteq f^{s-t}(f^t(X)) \supseteq f^{2(s-t)}(f^t(X)) \supseteq \ldots$ must stabilize.
(c) $\Rightarrow$ (d): $\Orb_f(X) = X \cup f(X) \cup \ldots \cup f^{s-1}(X)$ is a finite union of closed sets.
(b) $\Rightarrow$ (a): 
Without loss of generality, $f$ is nonconstant and $f(X) = X$.  
Then $f(\Irr(X)) \subseteq \Irr(X)$, 
for if $Z \in \Irr(X)$ then $f(Z)$ is an irreducible subset of $X$, hence contained in an irreducible component $W$ thereof; but $\dim f(Z) = \dim W$ by purity, so $f(Z) = W$. 
Since $\Irr(X)$ is finite, every $Z$ is preperiodic.
(d) $\Rightarrow$ (b): 
If $f^s(X) \ne f^t(X)$ for all $s \ne t$, 
then 
$\dim \overline{\Orb_f(X)} > \dim X$ (cf.~Remark \ref{rmk:stress}).
By the Baire category theorem \cite[p.~76, Exercise 5.10]{QingLiu},
$\Orb_f(X) \ne \overline{\Orb_f(X)}$.
\end{proof}

\begin{rmk}
Purity is necessary for (b) $\Rightarrow$ (a).
To wit, 
take $X = C \cup \{P\}$ where $P \not \in C \cup \PrePer(f)$, $f(P) \in C$, and $f(C) = C$
(e.g.,~$f(x : y : z) = (x^2 : y^2 : z^2)$, 
	$C = V(y^2 - xz)$, 
	and $P = (1 : t : -t^2)$ for $t$ not a root of unity).
Then $f(X) = f^2(X)$ 
yet $\{P\} \in \Irr(X)$ is not preperiodic.
\end{rmk}

\begin{propn}
Let $Z \subseteq \P^n$ be closed and irreducible.
Then
\[
	\overline{\Orb_f(Z)} = Z \cup f(Z) \cup \ldots \cup f^{t-1}(Z) \cup Z^* \cup f(Z^*) \cup \ldots \cup f^{p-1}(Z^*)
\]
for some integers $t \ge 0$ and $p \ge 1$ and some \emph{periodic} closed irreducible $Z^* \subseteq \P^n$ 
with 
    the property that 
$f^i(Z) \subseteq \Orb_f(Z^*)$ if and only if $i \ge t$.
Moreover, $\dim Z^* \ge \dim Z$ with equality 
    if and only if 
$Z$ is preperiodic with tail $t$ and period $p$.
\end{propn}

\begin{rmk}
$\Orb_f(Z^*)$ coincides with the $\omega$-limit set $\omega(Z, f) := \displaystyle \lim_{i \to \infty} \overline{\Orb_f(f^i(Z))}$. 
\end{rmk}

\begin{proof}
Let $Y = \overline{\Orb_f(Z)}$. 
By continuity, $f(Y) \subseteq Y$.
Thus, 
the top-dimensional components of $Y$ are preperiodic: 
there are finitely many of them, 
and they form an $f$-invariant subset of $\Irr(Y)$. 
Now, every component of $Y$ 
contains \emph{some} iterate of $Z$, 
lest it be redundant.
So, there exists a periodic top-dimensional component $Z^*$ which subsumes an iterate of $Z$ \emph{no later than} any other periodic top-dimensional component
    $Z'$ does; i.e.,~such that if $Z' \supseteq f^i(Z)$ then $Z^* \supseteq f^j(Z)$ for some $j \not> i$.
Let $p$ be 
    the period of $Z^*$ 
and choose $t$ minimal subject to $f^t(Z) \subseteq Z^*$.
Since $f^i(Z) \subseteq Y$ for all $i$, 
and $\{Z^*, f(Z^*), f^2(Z^*), \ldots\} \subseteq \Irr(Y)$, 
we have 
\[
	Z \cup f(Z) \cup \ldots \cup f^{t-1}(Z) \cup \Orb_f(Z^*) \subseteq Y.
\]
On the other hand, $f$-invariance implies $f^i(Z) \subseteq \Orb_f(Z^*)$ for all $i \ge t$, so
\[
	\Orb_f(Z) \subseteq Z \cup f(Z) \cup \ldots \cup f^{t-1}(Z) \cup \Orb_f(Z^*).
\]
Taking closures
gives the desired equality.
Next, suppose $f^i(Z) \subseteq \Orb_f(Z^*)$
for some $i$. 
By irreducibility, 
$f^i(Z) \subseteq f^j(Z^*)$ for some $j$. 
By choice of $Z^*$, $i \ge t$.
Lastly, $\dim Z \le \dim Z^*$ because $f^t(Z) \subseteq Z^*$. 
If $\dim Z = \dim Z^*$ then $f^t(Z) = Z^*$ is periodic---so $Z$ is preperiodic with tail $t$ and period $p$.
The converse is clear.
\end{proof}

\section{Dynamical Improperness} \label{sec:def} \label{sec:players}

\begin{defn} \label{defn:improperness}
Let $f$ be a 
nonconstant 
endomorphism of $\P^n$.
A hypersurface $H$ of $\P^n$ is \emph{improper} under $f$, 
or \emph{$f$-improper}, 
if for every irreducible component $Z$ of $H$ there exist $0 \le i_0 < \ldots < i_n$ such that $f^{i_0}(Z) \cap \ldots \cap f^{i_n}(Z) \ne \varnothing$.
We say $H$ is \emph{dynamically improper} if $H$ is $f^r$-improper 
for each $r > 0$.
\end{defn}

As alluded to in \S 1, the term ``improper'' refers to the fact that $n+1$ irreducible codimension-1 subvarieties of $\P^n$ intersect properly if and only if they have no point in common, while the term ``dynamical'' refers to 
invariance upon replacing $f$ by an iterate (cf.~Remark \ref{rmk:dynamicalization}).

One can show that $H$, $f(H)$, and $f^{-1}(H)$ are all improper or dynamically improper whenever any one of them is, and that both notions behave as expected under conjugation.
These claims may be verified through repeated judicious application of the following Proposition.

\begin{propn} \label{propn:basic}
Let $f, \varphi, g$ be three nonconstant endomorphisms of $\P^n$ with $\varphi f = g \varphi$ and let $H \subset \P^n$ be an $f$-improper hypersurface. 
Then $\varphi(H)$ is $g$-improper.
\end{propn}

\begin{proof}
Let $Z$ be an irreducible component of $\varphi(H)$. Then there exists an irreducible component $W$ of $H$ such that $\varphi(W) = Z$. Since $H$ is improper under $f$, there exist $i_0 < \ldots < i_n$ such that $f^{i_0}(W) \cap \ldots \cap f^{i_n}(W) \ne \varnothing$. Applying $\varphi$ shows that $g^{i_0}(Z) \cap \ldots \cap g^{i_n}(Z) \ne \varnothing$.
\end{proof}

The next result is our sole source of dynamically improper hypersurfaces.

\begin{propn} \label{propn:coperiodic}
If every irreducible component of $H$ contains a preperiodic subvariety, 
then $H$ is dynamically improper.
\end{propn}
\begin{proof}
Dynamical improperness is preserved under finite unions,
so we may assume $H$ is irreducible.
Replacing $H$ by an iterate, 
we may assume $H$ contains a periodic subvariety.
Replacing $f$ by an iterate, 
we may assume $H$ contains a fixed subvariety.
Now the claim is trivial.
\end{proof}

An immediate application is that dynamically improper hypersurfaces are dense in every degree. 
Note that our proof does not make use of the density of $\PrePer(f)$, only its infinitude (an elementary consequence of B\'ezout's theorem).

\begin{cor} \label{thm:1}
Let $m \ge 1$ and let $f$ be a nonlinear endomorphism of $\P^n$.
Write $\F^n_m$ for the projective space of degree-$m$ homogeneous forms in $n+1$ variables. 
Then the set 
\[
    \{\Phi \in \F^n_m : V(\Phi) \text{ is dynamically improper under } f\}
\]
is Zariski dense in $\F^n_m$.
\end{cor}

Before we prove Corollary \ref{thm:1}, we make some useful definitions and observations.
By stars-and-bars, 
\begin{equation}\label{eq:starsbars}
    \dim \F^n_m = \binom{n+m}{m} - 1.
\end{equation}
Put 
\[\H^n_m = \{\Phi \in \F^n_m : \Phi \text{ is irreducible}\}.\]
As the complement of the union of finitely many images of multiplication maps, $\H^n_m$ is open. 
Note that $\H^n_1 = \F^n_1$ for all $n$, while $\H^n_m = \varnothing$ if and only if $n = 1$ and $m > 1$.
For each point $P$ of $\P^n$ let
\[H_P = \{\Phi \in \F^n_m : \Phi(P) = 0\}\]
be the set of degree-$m$ hypersurfaces passing through $P$. 
Observe that $H_P$ is in fact a hyperplane: 
it is the projectivization of the kernel of the evaluation functional.
Finally, let $\eta : (\P^1)^m \to \F^1_m$ be the so-called Vieta map, the morphism sending each $m$-tuple of points $P_i = (a_i : b_i)$ to (the equivalence class of) the product 
\[\eta(P_1, \ldots, P_m) = (b_1 x - a_1 y) \ldots (b_m x - a_m y).\]

\begin{proof}[Proof of Corollary \ref{thm:1}]
Let $S$ be the set in question,
and consider $n = 1$. 
By induction on $m$, 
the set $D = \PrePer(f)^m$ is dense in the multiprojective space $(\P^1)^m$.
Since the Vieta map 
is continuous and surjective, $\eta(D)$ is dense in $\F^1_m$.
But $\eta(D) \subseteq S$ 
by Proposition \ref{propn:coperiodic}.

Next consider $n > 1$.
Since $\PrePer(f)$ is infinite and the $H_P$'s are pairwise distinct with codimension 1, the set
\[D = \bigcup_{\mathclap{P \in \PrePer(f)}} \ H_P\]
is dense in $\F^n_m$.
By Proposition \ref{propn:coperiodic}, $D \cap \H^n_m \subseteq S$.
As $\F^n_m$ is irreducible and $\H^n_m$ is non-empty open, $D \cap \H^n_m$---and therefore $S$---is dense in $\F^n_m$.
\end{proof}

\begin{rmk}
In the absence of preperiodic points, improperness admits the following characterization: 
an irreducible hypersurface $H$ disjoint from $\PrePer(f)$ is improper if and only if 
$H$ contains a \emph{frame}: $n+1$ points $P_0, \ldots, P_n$ with 
overlapping orbits such that $f^s(P_k) = f^s(P_l)$ implies $k = l$. 
Indeed, if $P \in f^{i_0}(H) \cap \ldots \cap f^{i_n}(H)$ for some $i_0 < \ldots < i_n$ then there exist points $P_k$ in $H$ such that $f^{i_k}(P_k) = P$ for all $k$.
Evidently $\Orb_f(P_0) \cap \ldots \cap \Orb_f(P_n) \ne \varnothing$, and if $f^s(P_k) = f^s(P_l)$ 
then $f^{s+i_l}(P_k) = f^{s+i_l}(P_l) = f^s(P) = f^{s+i_k}(P_k)$.
But $P_k \not \in \PrePer(f)$ so $s+i_l = s+i_k$, whence $k = l$. 
The converse is clear: pick any $P$ in every $\Orb_f(P_i)$.
\end{rmk}

\begin{rmk} \label{rmk:dynamicalization}
It is convenient to formalize the notion of a dynamical property like so.
For each subset $P$ of a multiplicative semigroup $E$, define
\begin{align*}
	\sqrt{P} &= \{f \in E : f^s \in P \text{ for some } s > 0\}, \\
	P^\circ &= \{f \in E : f^r \in P \text{ for all } r > 0\}. 
\end{align*}
The exponent laws imply these operations are monotone and idempotent.
Clearly, \[\sqrt{P} \supseteq P \supseteq P^\circ.\]
If the first containment is an equality we say $P$ is \emph{radical}; 
if both containments are equalities we say $P$ is \emph{dynamical}.
For instance, ``$X$ is preperiodic'' is a dynamical property, 
while ``$(P_0, \ldots, P_n)$ is a frame'' is not (as $\bigcap_i \Orb_{f^r}(P_i) = \varnothing$ for all large $r$).

Reversing quantifiers shows $\sqrt{P^\circ} \subseteq \sqrt{P}^\circ$ for all $P$. Therefore, if $\sqrt{P} = P$ then
\[\sqrt{P^\circ} = P^\circ = (P^\circ)^\circ.\]
In other words, any radical property $P$ can be upgraded to a dynamical property by declaring $f$ to be ``dynamically $P$'' if every iterate of $f$ is $P$.  

Since improperness is radical (trivially), dynamical improperness is dynamical.
Likewise, post-critical improperness is radical 
(as $C_{f^r} = \bigcup_{i<r} f^{-i}(C_f) \supseteq C_f$ by the chain rule) 
so ``dynamical'' post-critical improperness is dynamical; 
and \emph{this} property---namely, that every iterate of $f$ be post-critically improper---is equivalent to $f$ being post-critically dynamically improper (also by the chain rule).
\end{rmk}

\section{Examples} \label{sec:examples}

First, we exhibit an improper hypersurface that is not dynamically improper.

\begin{eg} \label{eg:I_not_DI}
Let $f = (x : y/2 : -{z/3})$. 
Then $L = V(x + y + z)$ is improper under $f$ but not under $f^2$.
Indeed, $L$ contains $f^i(5 : -32 : 27)$ for $i = 0, 2, 3$, so $L$ is $f$-improper. 
On the other hand, as $f$ is invertible and $f^i(L) = V(x + 2^i y + (-3)^i z)$ for all $i$, 
$L$ is $f^2$-improper if and only if there exist integers $0 < i < j$ such that
the linear system 
\[
x + y + z = x + 2^{2i} y + (-3)^{2i} z = x + 2^{2j} y + (-3)^{2j} z = 0
\]
has a nontrivial solution.
If we let $a = 4^i$ and $b = 9^i$, and put
\[q(t) = \det \begin{pmatrix} 
1 & 1 & 1 \\ 
1 & a & b \\
1 & a^t & b^t \end{pmatrix}
= (a - 1) b^t - (b - 1) a^t + b - a \qquad (t\in\mathbf{R})\]
then it's easy to see that $q$ is differentiable with just one critical point, hence by Rolle's theorem 
has at most two roots. 
Since $q(0) = q(1) = 0$ we conclude that $q(t) \ne 0$ for all other $t$, in particular for $t = j/i > 1$. 
Thus, $L$ is not $f^2$-improper.
\end{eg}

By interpolating a sufficiently
    sparse 
subset of the orbit of any wandering point $P$, one can find a hypersurface $H$ that is improper under any given number of iterates of $f$, yet is not obviously dynamically improper.
(Explicitly, if $H$ passes through $f^{si}(P)$ for $i = 0, \ldots, n$ then $H$ is $f^r$-improper for all $r \mid s$.)
It seems difficult to prove that any particular $H$ is not dynamically improper under a given $f$;
indeed, we were not able to produce any satisfactory examples with $\deg f > 1$.

\begin{q}
Does every $f$ admit an improper hypersurface that is not dynamically improper?
\end{q}

Next,
imitating the construction of dynatomic polynomials in dimension 1,
we describe how to 
find all 
improper hypersurfaces of any degree, for any map, in any dimension, 
of any index.
Fix $n, d, m \ge 1$.
Let $f \in \End^n_d$ and $\Phi \in \F^n_m$. 
For each $i \ge 0$ let $f^i_*\Phi$ denote the defining form of $f^i(V(\Phi))$,
and for each index $\underline{i} = (i_0, \ldots, i_n)$ let $P_{m,\underline{i},f}$
be the Macaulay resultant of $f^{i_0}_*\Phi$, \ldots, $f^{i_n}_*\Phi$ (see \cite[Chapter 3]{Cox} for relevant properties of resultants).
Then $P_{m,\underline{i},f}$ is a homogeneous polynomial in the coefficients of $\Phi$
which vanishes if and only if 
\[f^{i_0}(V(\Phi)) \cap \ldots \cap f^{i_n}(V(\Phi)) \ne \varnothing.\]
In particular, if $\Phi$ is irreducible and $P_{m,\underline{i},f}(\Phi) = 0$ for some strictly increasing $\underline{i}$, then $V(\Phi)$ is improper under $f$. 

Generically, 
\[\deg P_{m,\underline{i},f} 
= m^n d^{(n-1)(i_0 + \ldots + i_n)}
  \big(d^{i_0} + \ldots + d^{i_n}\big)\]
for all $0 \le i_0 < \ldots < i_n$, 
by homogeneity of the resultant and the fact that $f^i_*\Phi$ has degree $m d^{i(n-1)}$ in the variables and degree $d^{in}$ in the coefficients.
When $n = m = 1$, the $P_{m,\underline{i},f}$'s are essentially the ``(generalized) period polynomials'' of \cite[Exercise 4.10]{SilvermanADS}.
The divisibility relations between various $P_{m,\underline{i},f}$'s are mysterious.

We now illustrate the above procedure in the simplest nontrivial case.

\begin{eg} \label{eg:squaring}
Let $f = (x^2 : y^2 : z^2)$ and let $\Phi = ax + by + cz$ be a generic linear form. 
Using Sage and Macaulay2, we find that
\begin{align*}
    P_{1,(0,1,2),f}(a, b, c) 
    &= \Res(\Phi, f_*\Phi, f^2_*\Phi) \\
    &= a^8 b^8 c^8 (a+b)^2 (b+c)^2 (a+c)^2 (a+b+c)^2 \\
    &\phantom{==} (a^2 b + a b^2 + a^2 c + a c^2 + b^2 c + b c^2 - 6 abc) \\
    &\phantom{===} \Psi(a, b, c) \Psi(b, c, a) \Psi(c, a, b)
\end{align*}
where 
\begin{align*}
    f_*\Phi &= a^4 x^2 + b^4 y^2 + c^4 z^2 - 2(a^2 b^2 xy + a^2 c^2 xz + b^2 c^2 yz) \\ 
    f^2_*\Phi &= a^{16} x^4 + b^{16} y^4 + c^{16} z^4 \\
    &{} - 4(a^{12} b^4 x^3y + a^{12} c^4 x^3z + b^{12} a^4 y^3x + b^{12} c^4 y^3z + c^{12} a^4 z^3x + c^{12} b^4 z^3y) \\ 
    &{} + 6(a^8 b^8 x^2y^2 + b^8 c^8 y^2 z^2 + c^8 a^8 z^2x^2) \\
    &{} - 124(a^8 b^4 c^4 x^2yz + a^4 b^8 c^4 xy^2z + a^4 b^4 c^8 xyz^2)
\shortintertext{and}
    \Psi &= ac^6 + bc^6
+ a^2 b^5 + a^5 b^2  
+ a^5 c^2 + b^5 c^2 \\
&{} + 2(a^5 b c + a b^5 c 
+ a^4 b^2 c + a^2 b^4 c) 
+ 3(a^4 b^3 + a^3 b^4) \\
&{} - 4(a^4 c^3 + b^4 c^3 
+ a^2 b^2 c^3 
+ a^2 b c^4 + a b^2 c^4 
+ a^2 c^5 + b^2 c^5) \\
&{} - 5(a^4 b c^2 + a b^4 c^2)
+ 6(a^3 c^4 + b^3 c^4 + a^3 b c^3 + a b^3 c^3) \\
&{} + 13(a^3 b^2 c^2 + a^2 b^3 c^2).
\end{align*}

The linear factors of $P_{1,(0,1,2),f}$ parametrize lines containing certain fixed points of $f$, while each septic factor has genus 1 and is singular at 5 points.
The cubic factor is rational (i.e.,~birational to $\P^1$)
with parametrization 
\begin{align*}
    a(s, t) &= t (t-s) (s-2t) \\            
    b(s, t) &= a(t, s) \\
    c(s, t) &= st(t + s)
\end{align*}
Each line in this family contains the points 
$P := (s : -t : t - s)$, $f(P)$, $f^2(P)$, so is indeed improper.
Suppose that $ax + by + cz = 0$ for some $(x:y:z) \in \PrePer(f)$.
Since
\[
    \PrePer(f) \subseteq \{(x : y : z) : \text{each of $x, y, z$ is 0 or a root of unity}\},
\]
it follows that if $abc \ne 0$ and if none of $a/b$, $b/c$, $c/a$ is a root of unity, 
then $xyz \ne 0$, so $|x| = |y| = |z| = 1$.
Assuming $a, b, c \in \bar{\Q}$, then $-ax = by + cz$ \textit{et cetera} imply
\[
    |a| \le |b| + |c| \quad\text{and}\quad |b| \le |c| + |a| \quad\text{and}\quad |c| \le |a| + |b|.
\]
For our parametrized family, $abc = s^2 t^2 (s - t)^2 (s + t)(s - 2t)(2s - t)$,
so picking $(s, t) = (1, -2)$ yields $(a, b, c) = (30, -12, 2)$. 
Since $\lvert 30 \rvert > \lvert {-12} \rvert + \lvert 2 \rvert$ and none of $30/12$, $12/2$, $30/2$ is a root of unity, $L := V(30x - 12y + 2z)$ is disjoint from $\PrePer(f)$.
Thus, $L$ contains no preperiodic subvariety.
\end{eg}

Our third Example, based on \cite[Example 13]{IRS}, exhibits maps that are post-critically dynamically improper but not post-critically finite.

\begin{eg} \label{eg:PCDI_not_PCF}
Let $f = (x_0^d + c x_1^d : x_1^d : \ldots : x_n^d)$ for $n, d \ge 2$. 
Then $C_f = V(x_0 \ldots x_n)$ is always dynamically improper, but only sometimes preperiodic.
Indeed, $V(x_i)$ is fixed by $f$ for $i > 0$, 
while $V(x_0)$ contains the fixed point $(0 : 0 : \ldots : 1)$. 
However,
\[f(V(x_0 - zx_1)) = V(x_0 - (z^d + c)x_1)\]
for all $z$.
Thus $C_f$ is preperiodic under $f$ if and only if $0$ is preperiodic under $z^d + c$.
\end{eg}

The last Example in this Section is a toy model of some of the ideas behind the proof of the Main Theorem.

\begin{eg} \label{eg:random_family}
Let $f = (ax^2 + \alpha yz : by^2 + \beta xz : cz^2 + \gamma xy)$.
Then $f$ is a morphism if and only if 
\[\Res f = abc(abc + \alpha\beta\gamma)^3 \ne 0.\]
Trivially, $\Fix(f) \supseteq \{(1:0:0), (0:1:0), (0:0:1)\}$.
The Jacobian of $f$ is
\[
    J_f = \det \begin{pmatrix}2ax & \alpha z & \alpha y \\ \beta z & 2by & \beta x \\ \gamma y & \gamma x & 2cz\end{pmatrix} = 2\big((4abc + \alpha \beta \gamma)xyz - (a\beta\gamma x^3 + \alpha b\gamma y^3 + \alpha\beta c z^3)\big).
\]
By computing the resultant of the gradient of $J_f$, 
we see that $C_f = V(J_f)$ is smooth (hence, by the Leibniz rule, irreducible) if and only if 
\begin{equation} \label{eq:toy_eg}
\Phi := \Res \grad J_f = 2^{12} 3^3 (\alpha\beta\gamma)^2 (8abc - \alpha\beta\gamma)^6 \Res f \ne 0.     
\end{equation}
Assume henceforth that \eqref{eq:toy_eg} holds.
Then $C_f$ is disjoint from $\{(1:0:0), (0:1:0), (0:0:1)\}$, and $C_f \cap \Fix(f) \ne \varnothing$ if and only if the system 
\begin{align*}
(4abc + \alpha \beta \gamma)xyz &= a\beta\gamma x^3 + \alpha b\gamma y^3 + \alpha\beta c z^3 \\
x (by^2 + \beta xz) &= y(ax^2 + \alpha yz)  \\
x (cz^2 + \gamma xy) &= z(ax^2 + \alpha yz)  \\
y (cz^2 + \gamma xy) &= z(by^2 + \beta xz) 
\end{align*}
of 4 homogeneous equations in 3 unknowns has a nonzero solution. 
Notably, every nonzero solution $(x, y, z)$ to the first 3 equations is necessarily a solution to the 4\textsuperscript{th}, as $xyz = 0$ implies $x = y = z = 0$. 
The resultant of the first 3 equations is
\[
2^9 b c \alpha^8 \beta^4 \gamma^4 \Res f \cdot \Psi(a,b,c,\alpha,\beta,\gamma)
\]
where $\Psi$ has degree 15.
It follows that if $f \in V(\Psi) \setminus V(\Phi)$ then $f$ is a morphism whose critical locus is \emph{a fortiori} irreducible and contains a fixed point. 
In particular, such an $f$ is post-critically dynamically improper.
We note that by \cite[Theorem 4.1]{BD}, which says no elliptic curve can be critical and invariant for a nonlinear endomorphism of the plane, such an $f$ is not post-critically periodic.
\end{eg}

\section{Main Theorem: Reduction step} \label{sec:reduction}

Let $\bar{\End}^n_d$ denote the projective space of $(n+1)$-tuples of degree-$d$ homogeneous forms in $n+1$ variables, whose elements $(f_0 : \ldots : f_n)$ correspond bijectively to rational maps $f : \P^n \dashrightarrow \P^n$ of degree at most $d$.
Let 
\[\End^n_d = \{f \in \bar{\End}^n_d : f \text{ is a morphism of degree }d\}.\]
As the complement of the resultant locus, $\End^n_d$ is open,
so by \eqref{eq:starsbars} we have
\begin{equation} \label{eq:dimEnd}
    \dim \End^n_d = \dim \bar{\End}^n_d = (n+1)\binom{n+d}{d} - 1.
\end{equation}
For any $f$ in $\End^n_d$, 
let $C_f = V(J_f)$ denote the critical locus of $f$, where 
\[
    J_f(x_0, \ldots, x_n) = \det\frac{\partial f_i}{\partial x_j} 
    = \left| \begin{matrix} 
    \displaystyle \frac{\partial f_0}{\partial x_0} & \cdots & \displaystyle \frac{\partial f_0}{\partial x_n} \\
    \vdots & \ddots & \vdots \\
    \displaystyle \frac{\partial f_n}{\partial x_0} & \cdots & \displaystyle \frac{\partial f_n}{\partial x_n} 
    \end{matrix} \right|
\]
is the Jacobian determinant.
Since $J_f$ has degree $m = (n+1)(d-1)$,
it naturally induces a morphism $\End^n_d \to \F^n_m$ sending $f \mapsto J_f$.
The preimage of $\H^n_m$ under this morphism is the open set
\[U = \{f \in \End^n_d : C_f \text{ is irreducible}\}.\]

Our Main Theorem is twofold; we now show how the first part yields the second.

\begin{thm} \label{thm:new1}
Let $n \ge 1$ and $d \ge 2$. 
Then the set of maps with a periodic critical point is Zariski dense in $\End^n_d$.
\end{thm}

\begin{thm} \label{thm:new2}
Let $n \ge 1$ and $d \ge d_0$ where $d_0 = 2$ if $n < 4$ and $d_0 = 3$ otherwise. Then the set of post-critically dynamically improper maps is Zariski dense in $\End^n_d$.
\end{thm}

\begin{proof}[Proof of Theorem \ref{thm:new2}]
For $n = 1$ this is \cite[Theorem 1.6]{DeMarco}, so suppose $n \ge 2$.
Let $S$ be the locus in question
and let $D$ be the set of degree-$d$ endomorphisms of $\P^n$ 
with a periodic critical point. 
By Theorem \ref{thm:new1}, $D$ is dense in $\End^n_d$.
By \cite[Theorems 14 and 15(a)]{IRS} (and Remark \ref{rmk:U} below), $U$ is non-empty.
Since $\End^n_d$ is irreducible, $D \cap U$ is dense. 
But $D \cap U \subseteq S$ by Proposition \ref{propn:coperiodic}.
\end{proof}

\begin{rmk} \label{rmk:U} 
The results of Ingram--Silverman--Ramadas \cite{IRS} establishing $U \ne \varnothing$ are valid for all $d \ge \min\{n, 3\}$.
We expect $U \ne \varnothing$ for all $n, d \ge 2$. 
To prove this, we just need one $f \in \bar{\End}^n_2$ 
with $C_f$ irreducible. 
While it's easy to cook up matrices with simple determinants, not every matrix of linear forms is a Jacobian: by Euler's homogeneous function theorem, $f$ (viewed as a column vector) must satisfy 
\[ 
    \sum_{i=0}^n x_i \frac{\partial f}{\partial x_i}
    = 
    2f.
\]
In particular, if $f$ is not a morphism, 
then $C_f$ contains $f$'s indeterminacy locus.
Nor can $f$ be too simple: 
if $f_i = a_{i0} x_0^2 + \ldots + a_{in} x_n^2$ for some matrix $A = (a_{ij})$, 
then $J_f = 2^{n+1} x_0 \ldots x_n \det A$ by multilinearity.
Something like 
\[f_i = x_i^2 - c_i x_{i+1} x_{i+2}     \qquad (i = 0, \ldots, n \text{ taken mod } n+1)\]
might work: $f$ is a morphism iff $c_0 \ldots c_n \ne 1$, 
and when $n = 3$ and $(c_0, c_1, \ldots, c_n) = (-1, 1, \ldots, 1)$ 
then $C_f$ is irreducible because it's smooth. (Alas, we cannot rely on smoothness forever; it appears that $C_f$ is singular for all $f$ as soon as $n \ge 5$.)
\end{rmk}

The remainder of the paper is devoted to the proof of Theorem \ref{thm:new1}, which has two major parts.
The fundamental objects of study will be the sets
\[
    Y_s = \{f \in \End^n_d : \text{$f$ has a critical point of period dividing $s$}\} \qquad (s \ge 1)
\]
whose union is the locus in question.

\begin{lem} \label{lem:1}
For each $s \ge 1$, the set $Y_s$ is a hypersurface.
\end{lem}

\begin{lem} \label{lem:2}
For each pair of distinct prime numbers $p$ and $q$ exceeding $n$, the sets $Y_p$ and $Y_q$ are distinct.
\end{lem}

Lemmas \ref{lem:1} and \ref{lem:2} are the topics of Sections \ref{sec:elim} and \ref{sec:symm}, respectively.
Taking them for granted, Theorem \ref{thm:new1} becomes immediate.

\begin{proof}[Proof of Theorem \ref{thm:new1}]
Trivially, 
\[
    \bigcup_{p > n} Y_p \subseteq \bigcup_{s \ge 1} Y_s = \{f \in \End^n_d : \text{$f$ has a periodic critical point}\}.
\]
By Lemmas \ref{lem:1} and \ref{lem:2}, 
\[\overline{\bigcup_{p > n} Y_p} = \End^n_d. \qedhere\]
\end{proof}

\begin{rmk} \label{rmk:stress}
We stress that neither Lemma \ref{lem:1} nor \ref{lem:2} in isolation can guarantee density of the union of the $Y_s$'s.
This is because in general, if $\{Y_i : i \in I\}$ is an infinite set of pure $r$-dimensional closed subsets of a Noetherian space $X$, then 
the dimension of $\overline{\bigcup_{i \in I} Y_i}$ is at least $r + 1$, but not necessarily any greater.
Only when $r = \dim X - 1$ can we 
conclude that the union is dense in $X$.
Purity is also essential, 
as the dimension of a topological space is merely the supremum of the dimensions of its irreducible components; 
if impure, the $Y_i$'s can differ in their low-dimensional parts.
\end{rmk}

\section{Elimination theory} \label{sec:elim}

Our proof of Lemma \ref{lem:1} uses an algebro-geometric generalization of the rank--nullity theorem known as the \emph{fibre dimension theorem} \cite[p.~49, Theorem 3]{REDBOOK}: 

\begin{propn}
If $\varphi : X \to Y$ 
is a dominant morphism of irreducible varieties, 
then every irreducible component of every fibre of $\varphi$ has dimension at least $\dim X - \dim Y$.
Moreover, equality holds over a non-empty open subset of $Y$.
\end{propn}

Let 
\[
    \End^n_d(s) = \{(f, P) \in {\End^n_d} \times \P^n : f^s(P) = P\}
\]
be the space of degree-$d$ endomorphisms of $\P^n$ with a marked point of period dividing $s$.
\emph{A priori} it is not clear that $\End^n_d(s)$ has the same dimension as $\End^n_d$, given that 
\[\End^n_d(s) = V(x_i f^s_j - x_j f^s_i : 0 \le i, j \le n)\]
is cut out by $\binom{n+1}{2} > n$ equations. (Here, $f^s_i$ denotes the $i$\textsuperscript{th} component of the $s$\textsuperscript{th} iterate of $f$.)

\begin{propn} \label{propn:purity}
For each $n \ge 1$, $d \ge 2$, and $s \ge 1$, the variety $\End^n_d(s)$ is pure of dimension $N := \dim {\End^n_d}$ (cf.~equation \eqref{eq:dimEnd} in ~\S\ref{sec:reduction}).
\end{propn}
\begin{rmk}
    This follows from \cite[Lemma 2.2(1)]{Fakhruddin2014}, but we give a self-contained direct argument for completeness' sake.
    Incidentally, Proposition 3.3 \textit{loc.~cit.}~implies that each $\End^n_d[s]$ is irreducible (square brackets denoting exact period).
\end{rmk}
\begin{proof}
Let $Z$ be an irreducible component of $\End^n_d(s)$.
The canonical projection $\pi_1 : \End^n_d(s) \to \End^n_d$ has zero-dimensional fibres $\pi_1^{-1}(f) \cong \Fix(f^s)$, 
so the same is true of its restriction $Z \to \pi_1(Z)$.
By the fibre dimension theorem, 
\[\dim Z \le \dim \pi_1(Z) + \dim \pi_1^{-1}(f) \le {\dim \End^n_d} + 0 = N.\]
On the other hand, $Z$ is \emph{locally} at least $N$-dimensional. 
Indeed, considering the open cover 
\[U_i = V(x_i)^c = \{(f, P) \in {\End^n_d} \times \P^n : x_i(P) \ne 0\} \qquad (i = 0, \ldots, n)\] 
we must have $Z \cap U_i \ne \varnothing$ for some $i$, 
which makes $Z \cap U_i$ an irreducible component of $\End^n_d(s) \cap U_i$.
But 
\begin{equation} \label{eq:End_cap_Ui}
\End^n_d(s) \cap U_i = V(x_i f^s_j - x_j f^s_i : 0 \le j \le n, j \ne i) \cap U_i 
\end{equation}
because if $x_i(P) \ne 0$ and $(x_i f^s_j - x_j f^s_i)(P) = 0$ for all $j \ne i$, 
then 
\[x_i f^s_j x_k = x_j f^s_i x_k = x_j f^s_k x_i\]
at $P$, 
so that $(x_k f^s_j  - x_j f^s_k)(P) = 0$ for all $j, k$.
With only $n$ defining equations, the r.h.s.~of \eqref{eq:End_cap_Ui} is an open subset of a simpler variety---one whose every irreducible component has codimension at most $n$.
Thus, $\dim Z = \dim(Z \cap U_i) \ge N$.
\end{proof}

The Jacobian determinant (cf.~\S\ref{sec:reduction}) can be viewed as a bihomogeneous form $J$ in the coordinates of $x$ and the coefficients of $f$. Its zero-locus $V(J) \subseteq {\End^n_d} \times \P^n$ is precisely the space of degree-$d$ endomorphisms of $\P^n$ with a marked critical point. 

\begin{propn} \label{propn:J}
Let $n \ge 1$ and $d \ge 2$.
Then $V(J)$ is an irreducible hypersurface.
\end{propn}
\begin{proof}
As the vanishing locus of a single polynomial, $V(J)$ is closed and pure. 
To establish irreducibility, we use the following classical criterion:
If $X$ is pure, $Y$ is irreducible, and $\pi : X \to Y$ 
has irreducible fibres all of the same dimension, then $X$ is irreducible.\footnote{\textit{Proof (sketch).} Each fibre is wholly contained in some component of $X$. From this and Chevalley's theorem on constructible sets, $Y$ is dominated by a unique component of $X$. The fibre dimension theorem eliminates the possibility that $X$ has more than one component.}

The automorphism group $G = \PGL_{n+1}$ of $\P^n$ induces an action on ${\End^n_d} \times \P^n$ by 
\[
    \varphi \cdot (f, P) = (\varphi f \varphi^{-1}, \varphi(P))
\]
with respect to which the canonical projection $\pi_2 : {\End^n_d}\times \P^n \to \P^n$ is equivariant.
Since $G$ acts transitively on $\P^n$, the restriction of $\pi_2$ to any $G$-invariant subset of ${\End^n_d} \times \P^n$ is surjective with pairwise isomorphic fibres.
Now, $V(J)$ is $G$-invariant, because $J_{\varphi f \varphi^{-1}}(\varphi(P)) = \det \varphi \cdot J_f(P) \cdot \det \varphi^{-1}$ by the chain rule.\footnote{This identity holds literally only if the lift one chooses for $\varphi^{-1}$ is the inverse of the lift for $\varphi$.}
So, to prove the claim it suffices to show irreducibility of the fibre over the convenient point $P_0 = (1 : 0 : \ldots : 0)$.

To that end, we 
identify the fibre with a subset of
$\End^n_d$ and use coordinates $f_{i,\alpha}$ ($i = 0, \ldots, n$; $|\alpha| = \alpha_0 + \ldots + \alpha_n = d$). 
Let $e_i \in \Z^{n+1}$ denote the $i$\textsuperscript{th} standard unit vector.
The calculations 
\[
\frac{\partial f_i}{\partial x_0}(1, 0, \ldots, 0) 
= \sum_{|\alpha|=d} \alpha_0 f_{i,\alpha}  1^{\alpha_0-1} 0^{\alpha_1} \ldots 0^{\alpha_n} = d \cdot f_{i,(d, 0, \ldots, 0)} 
\]
and 
\[
\frac{\partial f_i}{\partial x_j}(1, 0, \ldots, 0) 
= \sum_{|\alpha|=d} \alpha_j f_{i,\alpha}  1^{\alpha_0} 0^{\alpha_1} \ldots 0^{\alpha_j-1} \ldots 0^{\alpha_n} = f_{i,(d-1,0,\ldots,1,\ldots,0)}
\]
for $j > 0$
imply 
$\pi_2|_{V(J)}^{-1}(P_0) \cong V(d \det(f_{i,(d-1)e_0+e_j}))$. 
Thus, the fibre is given by a generic determinant in $(n+1)^2$ indeterminates, so it's irreducible of dimension $N - 1$.
Thus $V(J)$ is irreducible, and 
\[
    \dim V(J) 
    = \dim \pi_2(V(J)) + \dim \pi_2|_{V(J)}^{-1}(P_0) 
    = n + N - 1. \qedhere
\]
\end{proof}

\begin{rmk}
The same argument can be used to 
show that $\End^n_d(1)$ is irreducible. In this case, the fibre $\pi_2|_{\End^n_d(1)}^{-1}(P_0) \cong V(f_{1,de_0}, \ldots, f_{n, de_0})$ is a linear subvariety.
\end{rmk}

\begin{proof}[Proof of Lemma \ref{lem:1}]
Unraveling definitions,
\[
    Y_s = \pi_1(\End^n_d(s) \cap V(J)).
\]
By the main theorem of elimination theory, 
the canonical projection $\pi_1 : {\End^n_d} \times \P^n \to \End^n_d$ is a closed map. 
Thus, $Y_s$ is closed.
Let $W$ be an irreducible component of $Y_s$. 
Then $W = \pi_1(Z)$ for some irreducible component $Z$ of $\End^n_d(s) \cap V(J)$.
The fibre dimension theorem gives $\dim W \ge \dim Z$,
and Propositions \ref{propn:purity} and \ref{propn:J} further entail
\[
    \codim Z \le \codim \End^n_d(s) + \codim V(J) = (N + n - N) + 1 = n + 1
\]
whence \textit{a fortiori} $\dim W \ge N + n - (n + 1) = N - 1$.
Since $W$ was arbitrary, 
it remains to prove $Y_s \ne {\End^n_d}$.

First consider the case $n = 1$.
Here, the bicritical map
\[
    f(z) = 1 + \frac{\zeta - 1}{z^d} \qquad (\zeta^d = 1, \ \zeta \ne 1)
\]
is totally ramified at $0$ and $\infty$, and sends 
\[
    0 \to \infty \to 1 \to \zeta \circlearrowleft
\]
so that neither of $f$'s critical points is periodic. 
Thus $f \not\in \bigcup_{s \ge 1} Y_s$, so $Y_s \ne \End^1_d$ for all $s$.
The general case follows from this and Proposition \ref{propn:key} below.
\end{proof}

\section{Symmetric powers} \label{sec:symm}

The Vieta map $\eta : (\P^1)^n \to \F^1_n$ (cf.~\S\ref{sec:players}) is the quotient map for the natural action of the symmetric group \cite[Corollary 2.6]{Maakestad}.
It follows that every rational map $f : \P^1 \to \P^1$ induces a morphism $F = s_n(f) : \F^1_n \to \F^1_n$ of the same degree, called its \emph{$n$\textsuperscript{th} symmetric power}, satisfying 
\begin{equation} \label{eq:Feta}
    F(\eta(P_1, \ldots, P_n)) = \eta(f(P_1), \ldots, f(P_n))
\end{equation}
for all $P_i \in \P^1$.
(Typically, one picks an isomorphism $\F^1_n \cong \P^n$.) Symmetric powers first appeared in the context of dynamics in the work of Ueda \cite[Section 4]{Ueda}.
A useful consequence of \eqref{eq:Feta} is that, with $H_P$ as in \S\ref{sec:players},
\begin{equation} \label{eq:FHP}
    F(H_P) = H_{f(P)}
\end{equation}
for all $P$ in $\P^1$,
which immediately implies $s_n : \End^1_d \to \End^n_d$ is injective.
We now quote a key result describing the critical locus of a symmetric power. 

\begin{propn} \label{propn:symm}
Let $F$ be the $n$\textsuperscript{th} symmetric power of the rational map $f$. 
Then the critical locus of $F$ is 
    \[
        C_F = E \, \cup \bigcup_{P \in C_f} H_P
    \]
where $E = \{\eta(P_1, \ldots, P_n) : f(P_i) = f(P_j) \text{ for some } P_i \ne P_j\}$. (Note: $E$ is not closed, but $C_F$ is.)
Moreover, $F(\bar{E}) = \Delta = F(\Delta)$, where $\Delta$ (for \emph{discriminant}) is the branch locus of $\eta$.
\end{propn}
\begin{proof}[Proof (sketch).]
One applies the chain rule to \eqref{eq:Feta}; see \cite[Lemma 2.1]{GHK} or \cite[Example 10]{IRS} for details.
See also \cite[Section 1]{Ingram} for a statement involving multiplicities. 
\end{proof}

\begin{eg}
$H_P$ is improper for $F$ if and only if $P \in \PrePer(f)$. 
Indeed, by \eqref{eq:FHP},
\[\bigcap_{k=0}^n F^{i_k}(H_P) = \bigcap_{k=0}^n H_{f^{i_k}(P)}\]
and this intersection is non-empty if and only if there exists $\Phi \in \F^1_n$ vanishing at each $f^{i_k}(P)$. 
But $\Phi$ has at most $n$ distinct roots, so the only way this can happen is if $f^{i_k}(P) = f^{i_l}(P)$ for some $k \ne l$, i.e.,~if $P$ is preperiodic.
\end{eg}

Applying this Example to the critical locus of $F$, we obtain the following partial generalization of \cite[Theorem 3]{GHK}.

\begin{cor} 
Post-critical improperness, post-critical dynamical improperness, and post-critical finiteness of $F$ are all equivalent to post-critical finiteness of $f$.
\end{cor}

Returning to the main thread, 
we proceed to describe the intersection of $Y_s$ with the locus of symmetric power maps.

\begin{propn} \label{propn:key}
Let $s_n : \End^1_d \to \End^n_d$ be the morphism sending $f$ to its $n$\textsuperscript{th} symmetric power.
Then 
\[
    s_n^{-1}(Y_s) = \! \bigcup_{\substack{t \mid ms \\ 1 \le m \le n}} \! Y_t
\]
for all $s \ge 1$.
(Note: the $Y_s$ on the left is a subset of $\End^n_d$ while the $Y_t$'s on the right are subsets of $\End^1_d$.)
\end{propn}

\begin{proof}
Let $f \in s_n^{-1}(Y_s)$. 
Then $F := s_n(f)$ has a critical point $\Phi = \eta(P_1, \ldots, P_n)$ of period $s$.
By \eqref{eq:Feta}, 
$f^s$ permutes the set $\{P_1, \ldots, P_n\}$;
in particular, each $P_i$ is periodic.
Since $f$ is injective on $\Per(f)$, $\Phi \not \in E$, so Proposition \ref{propn:symm} implies $\Phi \in H_P$ for some $P := P_1 \in C_f$, say.
Thus $f \in Y_t$ where $t$ 
is the period of $P$ under $f$.
Letting $m$ 
be the period of $P$ under $f^s$,
we have $t \mid ms$ because $f^{ms}(P) = (f^s)^m(P) = P$, 
and we have $1 \le m \le n$ because $\Orb_{f^s}(P) \subseteq \{P_1, \ldots, P_n\}$.

Conversely, suppose $f \in Y_t$ for some $t \mid ms$ with $1 \le m \le n$. 
Let $P \in C_f \cap \Fix(f^t)$.
Picking $Q \in \Fix(f)$ and setting $F = s_n(f)$ it is easy to check that
\[
\Phi = \eta(P, f^s(P), \ldots, f^{(m-1)s}(P), \underbrace{Q, \ldots, Q}_{n-m}) \in C_F \cap \Fix(F^s). \qedhere
\]
\end{proof}

\begin{rmk}
Using M\"obius inversion, 
one can show that the set of integers $t$ dividing 
$ms$ for some $1 \le m \le n$ 
has cardinality $\chi(s) n + \xi(n)$, where $\chi$ is a multiplicative function satisfying $\chi(p^k) = 1 + k(1-1/p)$, and $\xi$ has period $\operatorname{rad}(s)$.
\end{rmk}

\begin{proof}[Proof of Lemma \ref{lem:2}]
The argument hinges on the existence of a separating sequence of maps. 
Namely, we claim that for each $d \ge 2$ and each prime number $p$, 
there exists a parameter $c$ such that every critical point of the rational map 
\[f(z) = z^{-d} + c\]
has exact period $p$.
Indeed, 
$f$ is bicritical, being totally ramified at $0$ and $\infty$, and sends
\[0 \to \infty \to c\]
so we must solve the equation $f^p(0) = 0$ for $c$.
It can be shown (e.g.,~by induction) that for $s \ge 2$, the numerator of $f^s(0)$ has degree $1 + d + d^2 + \ldots + d^{s-2} \ge 1$ as a polynomial in $c$.
Taking $c$ to be any root (for $s = p$) proves the claim.

Performing this construction for each prime number $p$ yields a sequence of rational maps $f_p$ in $\End^1_d$ each of whose critical points all have exact period $p$.
Now suppose $s_n(f_p) \in Y_q$ for some primes $p, q > n$.
By Proposition \ref{propn:key}, 
$f_p \in Y_t$ for some $t \mid mq$ with $1 \le m \le n$.
The definition of $f_p$ implies $p \mid t$, 
which means $p \mid m$ or $p \mid q$. 
But $m$ is too small, so $p = q$.
In other words, $Y_p$ is the only set among the $Y_q$'s that contains $s_n(f_p)$. 
Thus $Y_p \ne Y_q$ for all primes $p \ne q > n$.
\end{proof}

\printbibliography

\end{document}